\documentclass[a4paper,10pt]{article}

\usepackage{amssymb, amsmath, amsthm, hyperref, graphicx}

\newtheorem{thm}{Theorem}
\newtheorem{prb}{Problem}

\title{How to decide whether two convex octahedra \\
       are affinely equivalent using their natural \\
       developments only}
\author{Victor Alexandrov \\ \textit{Sobolev Institute of Mathematics, Novosibirsk, Russia} \& \\ 
\textit{Department of Physics, Novosibirsk State University, Russia} \\ \texttt{alex@math.nsc.ru}}

\begin{document}
\maketitle

\vskip-5mm\hfill\begin{minipage}[t]{6.3 cm}
\textit{Dedicated to Professor Hellmuth Stachel on the occasion of his 80th birthday}
\end{minipage}

\begin{abstract}
Given two convex octahedra in Euclidean 3-space, we find conditions on their natural developments 
which are necessary and sufficient for these octahedra to be affinely equivalent to each other.
\end{abstract}

\section{Introduction}\label{s:1}

The question we shall address in this article is this:
\begin{prb}\label{prb:1}
Given two convex octahedra in Euclidean 3-space, is it possible to decide whether they are affinely equivalent 
to each other using their natural developments only?
\end{prb}

From our point of view, Problem \ref{prb:1} is similar to the problem of recognizing congruent convex polyhedra 
whose solution is given by the famous Cauchy rigidity theorem:
\textit{If natural developments of two convex closed polyhedra in Euclidean 3-space are isometric 
then these polyhedra are congruent to each other.} 
Initially, this theorem was proved by A.L.~Cauchy in 1813 in \cite{Ca13}.
An extensive literature is devoted to the Cauchy rigidity theorem and its generalizations, 
from which we mention the monograph \cite{Al05}, review article \cite{Co93}, scientific article \cite{RR00}, 
and popular science book \cite{FT07}, where the reader can find further references.

An analog of Problem~\ref{prb:1} for general polyhedra in Euclidean 3-space is considered in \cite{Al21},
where necessary conditions for the affine equivalence of polyhedra are found.
In this article, we present a similar approach as applied to octahedra.
This allows us to obtain both necessary and sufficient conditions.

The main result of this article is Theorem~\ref{thm:1} from Section~\ref{s:4}.

We choose octahedra as the object of our study for two reasons.
First, because they are the simplest polyhedra in terms of combinatorial structure with no trivalent vertices 
(the latter simplify the problem of recognition of affinely equivalent polyhedra 
by their natural developments, see~\cite{Al21}). 
Second, because historically octahedra played an intriguing role in the proof of the Cauchy rigidity theorem
(see~\cite[p.~446]{Sa11}) while our study is motivated by that theorem.

\section{Terminology and preliminaries}\label{s:2}

\subsubsection*{Octahedron}
We say that a polyhedral surface in~$\mathbb{R}^3$ is an \textit{octahedron} if it is combinatorially equivalent to the regular convex octahedron in~$\mathbb{R}^3$ and any two adjacent faces do not lie in the same plane.

A straight line segment $xy$ is called a \textit{diagonal} of an octahedron $P$ if 
$x$ and $y$ are vertices of $P$ and $xy$ is not an edge of $P$.

An octahedron $P$ is called \textit{convex} if, for every its face $F$, the three vertices of $P$ which are
not incident to $F$, are contained in a single open halfspace determined by $\mbox{aff\,} F$, the affine hull of $F$.

We leave the proofs of the following statements to the reader:
every convex octahedron $P$ is the boundary of a bounded convex set $B$ with nonempty interior, i.e., of a convex body;
moreover, $B$ contains every diagonal of $P$.

\subsubsection*{Natural development}
The concept of a natural development of an octahedron is intuitively obvious to anyone who has 
ever glued an octahedron from a set of its faces which are cut out from cardboard.
Avoiding unnecessary details, we give the following more or less formal definition.

A set consisting of 8 triangles in the plane is called the \textit{natural development}~$R$ of an octahedron~$P$
if it is equipped with ``gluing rules'' and the following conditions are fulfilled:

{\scriptsize$\bullet$} $R$ is in one-to-one correspondence with the set of the faces of~$P$;

{\scriptsize$\bullet$} each triangle of~$R$ is congruent to the corresponding face of~$P$;

{\scriptsize$\bullet$} vertices of different triangles of~$R$ are identified with each other (or ``are glued together'') 
if and only if they correspond to the same vertex of~$P$;

{\scriptsize$\bullet$} sides of different triangles of~$R$ are identified with each other (or ``are glued together'') 
if and only if they correspond to the same edge of~$P$.

\subsubsection*{Cayley--Menger determinant}
Let $M$ be an abstract set and $\rho:M\times M\to[0, +\infty)\subset\mathbb{R}$ be a map
such that $\rho(x,y)=\rho(y,x)$ for all $x,y\in M$ and $\rho(x,y)=0$ if and only if $x=y$.
Then the pair $(M,\rho)$ is called a \textit{semimetric space}, see \cite[Definition 5.1, p. 7]{Bl70}, and,
for every $\{x_0, x_1,\dots, x_k\}\subset M$, the expression
\begin{equation*}
\operatorname{cm}(x_0, x_1, \dots, x_k)\stackrel{\textrm{def}}{=}
\left|
\begin{array}{ccccc}
0 & 1           & 1            & \dots       & 1         \\
1 & 0           & \rho^2(x_0,x_1)     & \dots       & \rho^2(x_0,x_k)  \\
1 & \rho^2(x_1,x_0)    & 0            & \dots       & \rho^2(x_1,x_k)  \\
. & .           & .            & .           & .         \\
1 & \rho^2(x_k,x_0)    & \rho^2(x_k,x_1)     & \dots       & 0 
\end{array}
\right| 
\end{equation*}
is called the \textit{Cayley--Menger determinant} of $x_0, x_1,\dots, x_k$, see
\cite[\S~40, p. 97]{Bl70}.

If $M=\{x_0, x_1,\dots, x_n\}\subset \mathbb{R}^n$ and a metric $\rho$ in $M$ 
is induced by the Euclidean metric of $\mathbb{R}^n$
then $\operatorname{vol}(x_0, x_1, \dots, x_n)$, the $n$-dimensional volume of the simplex 
with the vertices $x_0, x_1, \dots, x_n$, is related to the Cayley--Menger determinant of 
$x_0, x_1, \dots, x_n$ by the following formula, see \cite[\S~40, p. 98]{Bl70}:
\begin{equation}
[\operatorname{vol}(x_0, x_1, \dots, x_n)]^2=
\frac{(-1)^{n+1}}{2^n n!}
\operatorname{cm}(x_0, x_1, \dots, x_n).
\label{eq:1}
\end{equation}

We say that a semimetric space $(M,\rho)$ \textit{embeds isometrically} in $\mathbb{R}^n$
if there is $f:M\to\mathbb{R}^n$ such that $|f(x)-f(y)|=\rho(x,y)$ for every $x,y\in M$, 
where $|f(x)-f(y)|$ stands for the Euclidean distance between $f(x), f(y)\in\mathbb{R}^n$.

Necessary and sufficient conditions for a semimetric space $(M,\rho)$ to embed isometrically 
in $\mathbb{R}^n$ are given by the following

\begin{thm}[K.~Menger, 1928]\label{thm:1}
For every semimetric space $M$, the following statements \emph{(i)} and \emph{(ii)} are equivalent to each other:
\par
\emph{(i)} $M$ embeds isometrically in $\mathbb{R}^n$, but not in $\mathbb{R}^{n-1}$;
\par
\emph{(ii)} there are $n+1$ points $x_0, x_1, \dots, x_n$ in $M$ such that 
\begin{equation}
(-1)^{k+1} \operatorname{cm}(x_0, x_1, \dots, x_k)>0
\label{eq:2}
\end{equation}
 for all $k=1,2,\dots, n$ and
\begin{alignat}{2}
&\operatorname{cm}(x_0, x_1, \dots, x_n, x) &\, =\, & 0,
\label{eq:3}\\
&\operatorname{cm}(x_0, x_1, \dots, x_n, y) &\, =\, & 0,
\label{eq:4}\\
&\operatorname{cm}(x_0, x_1, \dots, x_n, x,y) &\, =\, & 0 
\label{eq:5}
\end{alignat}
for every pair of points $x,y$ of $M$. 
\end{thm} 

The classic reference for Theorem~\ref{thm:1} is~\cite[Theorem 42.2, p. 104]{Bl70}.
A modern exposition of Theorem 1 for the case of metric spaces may be found in ~\cite{BB17}.

\section{Necessary conditions}\label{s:3}
Throughout this Section, we denote by $P$ and $P'$ two fixed octahedra in $\mathbb{R}^3$.

By $x_i$, $i=0,\dots, 5$, denote the vertices of $P$.
Throughout this article, we fix their special enumeration, an example of which is shown in Fig.~\ref{fig:1}.
Such a special enumeration is obtained in the following way.
Chose an edge $e$ of $P$ arbitrarily and denote by $x_0$ and $x_1$ the vertices of $P$ incident to $e$. 
For $i=2,3$, by $x_i$ denote a vertex of $P$ incident to a face of $P$ containing $e$.
By $x_4$ (resp., $x_5$) denote a vertex of $P$ such that the straight line segment $x_1x_4$ 
(resp., $x_0x_5$) is a diagonal of $P$.
\begin{figure}[!hbt]
\centering
\includegraphics[width= 35 mm ]{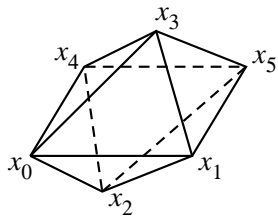}
\caption{Special enumeration of the vertices of an octahedron~$P$}
\label{fig:1}
\end{figure}

By $\langle x_i, x_j\rangle$ denote the edge of $P$ incident to the vertices  $x_i, x_j$ if such an edge exists.
Similarly, by $\langle x_i, x_j, x_k\rangle$ denote the face of $P$ incident to the vertices  $x_i, x_j, x_k$ 
if such a face exists.

Consider the set $V=\{x_0,\dots,x_5\}$ endowed with a metric induced from $\mathbb{R}^3$.
For the convenience of further presentation, we denote the Euclidean distance 
$|x_i-x_j|$ between $x_i, x_j\in V$ by $d_{ij}$ if $x_i$ and $x_j$ are joined by an edge of $P$, 
and by $\delta_{ij}$ otherwise.
In other words, we denote $|x_i-x_j|$ by $d_{ij}$ if this distance is given to us directly in 
the natural development of $P$, and denote it by $\delta_{ij}$ otherwise.
We choose different notations $d_{ij}$ and $\delta_{ij}$ because, when solving Problem~\ref{prb:1} in 
Section~\ref{s:4}, we consider $\delta_{ij}$ as unknown quantities, the values of which have to be found 
in the course of the solution.

Throughout Section~\ref{s:3}, we denote by $x'_i$, $i=0,\dots, 5$, the vertices of $P'$
and assume that~$P'$ is affinely equivalent to $P$ in the sense that there exists an affine 
transformation $A:\mathbb{R}^3\to\mathbb{R}^3$ such that $A(P)=P'$ and $A(x_i)=x'_i$ for all $i=0,\dots, 5$, 
i.\,e., that $A$ does not change the numbering of the vertices.
The Euclidean distance between the vertices $x'_i$ and $x'_j$ of $P'$ is denoted by $d'_{ij}$ 
if $x'_i$ and $x'_j$ are connected by an edge of $P'$ and by $\delta'_{ij}$ otherwise.

\subsubsection*{The first group of necessary conditions}
Directly by its construction, the set $V=\{x_0,\dots,x_5\}$ endowed with the metric induced from $\mathbb{R}^3$ 
is a metric space which embeds isometrically in $\mathbb{R}^ 3$, but not in $\mathbb{R}^ 2$.
Hence, by Theorem~\ref{thm:1}, the inequalities (\ref{eq:2}) and equalities (\ref{eq:3})--(\ref{eq:5}) hold.
Let us write them down in detail.

The inequalities (\ref{eq:2}) take the form 
\begin{alignat*}{2}
&(-1)^2\operatorname{cm}(x_0, x_1) & \,>\, & 0,\\
&(-1)^3\operatorname{cm}(x_0, x_1, x_2) & \,>\, & 0,\\
&(-1)^4\operatorname{cm}(x_0, x_1, x_2, x_3)&\, >\, & 0 ,
\end{alignat*}
i.e.,
\begin{equation}
d^2_{01}>0, \qquad 
\left|
\begin{array}{cccc}
0 & 1         & 1              & 1         \\
1 & 0         & d^2_{01}       & d^2_{02}  \\
1 & d^2_{01}  & 0              & d^2_{12}  \\
1 & d^2_{02}  & d^2_{12}       & 0         \\ 
\end{array}
\right|<0,
\label{eq:6}
\end{equation}
\begin{equation}
\left|
\begin{array}{ccccc}
0 & 1         & 1              & 1               & 1               \\
1 & 0         & d^2_{01}       & d^2_{02}        & d^2_{03}        \\
1 & d^2_{01}  & 0              & d^2_{12}        & d^2_{13}        \\
1 & d^2_{02}  & d^2_{12}       & 0               & \delta^2_{23}   \\
1 & d^2_{03}  & d^2_{13}       & \delta^2_{23}   & 0              
\end{array}
\right|>0.
\label{eq:7}
\end{equation}
The inequalities (\ref{eq:6}) mean that $x_0\neq x_1$ 
and the area of the face $\langle x_0, x_1, x_2\rangle$ is not equal to zero.
The inequality (\ref{eq:7}) means that the volume of the tetrahedron $\langle x_0, x_1, x_2, x_3\rangle$ 
with the vertices $x_i$, $i=0,\dots, 3$ is nonzero.
According to our definition of an octahedron, (\ref{eq:6})--(\ref{eq:7}) are automatically satisfied.

Note that if we write (\ref{eq:7}) using only numerical data that is available 
to us from the natural development of~$P$ (i.e., using $d_{ij}$ only) then (\ref{eq:7}) imposes a
restriction on the possible values of the length $\delta_{23}$ of the diagonal $x_2x_3$ of $P$.

The equality (\ref{eq:3}), when written for~$P$, becomes $\operatorname{cm}(x_0, x_1, x_2, x_3, x_4){=}0$, i.e.,
\begin{equation}
\left|
\begin{array}{cccccc}
0 & 1         & 1              & 1               & 1               & 1              \\
1 & 0         & d^2_{01}       & d^2_{02}        & d^2_{03}        &  d^2_{04}      \\
1 & d^2_{01}  & 0              & d^2_{12}        & d^2_{13}        & \delta^2_{14}  \\
1 & d^2_{02}  & d^2_{12}       & 0               & \delta^2_{23}   & d^2_{24}       \\
1 & d^2_{03}  & d^2_{13}       & \delta^2_{23}   & 0               & d^2_{34}       \\
1 & d^2_{04}  & \delta^2_{14}  & d^2_{24}        & d^2_{34}        & 0
\end{array}
\right| =0.
\label{eq:8}
\end{equation}

The equality (\ref{eq:4}) takes the form $\operatorname{cm}(x_0, x_1, x_2, x_3, x_5)=0$, i.e.,
\begin{equation}
\left|
\begin{array}{cccccc}
0 & 1              & 1          & 1               & 1               & 1              \\
1 & 0              & d^2_{01}   & d^2_{02}        & d^2_{03}        & \delta^2_{05}  \\
1 & d^2_{01}       & 0          & d^2_{12}        & d^2_{13}        & d^2_{15}       \\
1 & d^2_{02}       & d^2_{12}   & 0               & \delta^2_{23}   & d^2_{25}       \\
1 & d^2_{03}       & d^2_{13}   & \delta^2_{23}   & 0               & d^2_{35}       \\
1 & \delta^2_{05}  & d^2_{15}   & d^2_{25}        & d^2_{35}        & 0
\end{array}
\right| =0.
\label{eq:9}
\end{equation}

Finally, the equality (\ref{eq:5}) becomes $\operatorname{cm}(x_0, x_1, x_2, x_3, x_4, x_5)=0$, i.e.,
\begin{equation}
\left|
\begin{array}{ccccccc}
0 & 1              & 1              & 1               & 1               & 1              & 1             \\
1 & 0              & d^2_{01}       & d^2_{02}        & d^2_{03}        & d^2_{04}       & \delta^2_{05} \\
1 & d^2_{01}       & 0              & d^2_{12}        & d^2_{13}        & \delta^2_{14}  & d^2_{15}      \\
1 & d^2_{02}       & d^2_{12}       & 0               & \delta^2_{23}   & d^2_{24}       & d^2_{25}      \\
1 & d^2_{03}       & d^2_{13}       & \delta^2_{23}   & 0               & d^2_{34}       & d^2_{35}      \\
1 & d^2_{04}       & \delta^2_{14}  & d^2_{24}        & d^2_{34}        & 0              & d^2_{45}      \\
1 & \delta^2_{05}  & d^2_{15}       & d^2_{25}        & d^2_{35}        & d^2_{45}       & 0
\end{array}
\right| =0.
\label{eq:10}
\end{equation}

The relations (\ref{eq:7})--(\ref{eq:10}) form the first group of necessary conditions for the affine equivalence 
of the octahedra $P$ and $P'$.
They must be satisfied by the quantities $\delta_{05}$, $\delta_{14}$, and $\delta_{23}$, 
which are the lengths of the diagonals of $P$ and which cannot be found directly from
the natural development of~$P$.

\subsubsection*{The second group of necessary conditions}
This group of conditions expresses the fact that $P$ is convex.

Let us start with a detailed explanation of the construction of one of the conditions of the second group of necessary conditions.
To do this, consider two tetrahedra $T_3=\langle x_0, x_1, x_2, x_3\rangle$ and 
$T_4=\langle x_0, x_1, x_2, x_4\rangle$ in $\mathbb{R}^3$.
They share the common face $\langle x_0, x_1, x_2\rangle$.
In principle, $T_3$ and $T_4$ can be located relative to each other in two ways: 
either so that they are contained in different closed halfspaces determined by the plane containing 
the face $\langle x_0, x_1, x_2\rangle$, or so that they are contained in one of these halfspaces.
Obviously, $|x_3-x_4|$ (and hence the numerical value of $d^2_{34}$) in the latter case is strictly less 
than in the first one.

In (\ref{eq:8}), substitute $t$ instead of $d^2_{34}$ (the other $d_{ij}$'s and $\delta_{14}$ are 
assumed to be fixed numbers borrowed from $P$).
As a result, we get
\begin{equation}
\left|
\begin{array}{cccccc}
0 & 1         & 1              & 1               & 1               & 1              \\
1 & 0         & d^2_{01}       & d^2_{02}        & d^2_{03}        &  d^2_{04}      \\
1 & d^2_{01}  & 0              & d^2_{12}        & d^2_{13}        & \delta^2_{14}  \\
1 & d^2_{02}  & d^2_{12}       & 0               & t               & d^2_{24}       \\
1 & d^2_{03}  & d^2_{13}       & t               & 0               & d^2_{34}       \\
1 & d^2_{04}  & \delta^2_{14}  & d^2_{24}        & d^2_{34}        & 0
\end{array}
\right| =0.
\label{eq:11}
\end{equation}
Let us rewrite (\ref{eq:11}) in the form $At^2+Bt+C=0$, where $A$, $B$ and $C$ are algebraic polynomials in 
$\delta_{14}$ and all variables $d_{ij}$ involved in (\ref{eq:11}). 
It is easy to see that
\begin{equation*}
A=-\left|
\begin{array}{cccc}
0 & 1         & 1               & 1              \\
1 & 0         & d^2_{01}        &  d^2_{04}      \\
1 & d^2_{01}  & 0               & \delta^2_{14}  \\
1 & d^2_{04}  & \delta^2_{14}   & 0
\end{array}
\right|=-\operatorname{cm}(x_0, x_1, x_4).
\end{equation*}
Hence, taking into account the formula (\ref{eq:1}), we get $A={2^2}{2!}[\operatorname{vol}(x_0, x_1, x_4)]^2>0$.
Here $\operatorname{vol}(x_0, x_1, x_4)$ denotes the area of the triangle with the vertices $x_0, x_1, x_4$, and
the inequality holds true because $\operatorname{vol}(x_0, x_1, x_4)\neq 0$.
The latter follows form the definition of an octahedron given in Section~\ref{s:2}.

Above, we mentioned a geometric argument based on the relative position of $T_3$, $T_4$.
It implies that the quadratic equation $At^2+Bt+C=0$ has two different positive real roots.
Let us denote them by $t_1$, $t_2$.
For definiteness, assume $0<t_1<t_2$. Then $t_1=d^2_{34}$,
$B^2-4AC>0$, $B=-A(t_1+t_2)< -2At_1$, and  $C=At_1t_2> At_1^2$.
For positive $t_1,t_2$, each of the last two inequalities is equivalent to $t_1<t_2$
and, thus, they are equivalent to each other.
Hence, below we can use only one of them, e.g., $C > At_1^2$ (or, what is the same, $C> Ad^4_{34}$).

Since $C=At^2+Bt+C$ for $t=0$ and $At^2+Bt+C$ is given by the formula (\ref{eq:11}), 
we can rewrite the inequality $C> Ad^4_{34}$ as 
\begin{equation}
\left|
\begin{array}{cccccc}
0 & 1         & 1              & 1               & 1               & 1              \\
1 & 0         & d^2_{01}       & d^2_{02}        & d^2_{03}        &  d^2_{04}      \\
1 & d^2_{01}  & 0              & d^2_{12}        & d^2_{13}        & \delta^2_{14}  \\
1 & d^2_{02}  & d^2_{12}       & 0               & 0               & d^2_{24}       \\
1 & d^2_{03}  & d^2_{13}       & 0               & 0               & d^2_{34}       \\
1 & d^2_{04}  & \delta^2_{14}  & d^2_{24}        & d^2_{34}        & 0
\end{array}
\right| > -d^4_{34}
\left|
\begin{array}{cccc}
0 & 1         & 1               & 1              \\
1 & 0         & d^2_{01}        &  d^2_{04}      \\
1 & d^2_{01}  & 0               & \delta^2_{14}  \\
1 & d^2_{04}  & \delta^2_{14}   & 0
\end{array}
\right|.
\label{eq:12}
\end{equation}
The inequality~(\ref{eq:12}) expresses the fact that the points $x_3$, $x_4$ are contained in one of
the two closed halfspaces determined by the plane containing the face $\langle x_0, x_1, x_2\rangle$ of $P$.
If we write (\ref{eq:12}) using only distances given to us in the natural development of $P$
(i.e., using $d_{ij}$'s only) 
then (\ref{eq:12}) imposes restrictions on the possible values of the length of the diagonal $x_1x_4$
(i.e., on $\delta_{14}$).

The inequality (\ref{eq:12}) belongs to the second group of necessary conditions for the affine equivalence 
of the octahedra $P$ and $P'$.
The other inequalities of that group are derived similarly to (\ref{eq:12}) according to the following instructions:

{\scriptsize$\bullet$} select one of the eight faces of $P$ and denote it by $\langle x_{i_0}, x_{i_1}, x_{i_2}\rangle$
(it will play the same role as the face $\langle x_0, x_1, x_2\rangle$ played in the above derivation 
of~(\ref{eq:12}));

{\scriptsize$\bullet$} select two of the three faces of $P$ incident to $\langle x_{i_0}, x_{i_1}, x_{i_2}\rangle$
and denote them by $\langle x_{i_0}, x_{i_1}, x_{i_3}\rangle$ and
$\langle x_{i_0}, x_{i_2}, x_{i_4}\rangle$ (they will play the same role as the faces 
$\langle x_0, x_1, x_3\rangle$ and $\langle x_0, x_2, x_4\rangle$ played in the above derivation of~(\ref{eq:12}));

{\scriptsize$\bullet$} apply the arguments described above for tetrahedra 
$T_3=\langle x_0, x_1, x_2, x_3\rangle$ and $T_4=\langle x_0, x_1, x_2, x_4\rangle$
to tetrahedra $\langle x_{i_0}, x_{i_1}, x_{i_2}, x_{i_3}\rangle$ and $\langle x_{i_0}, x_{i_1}, x_{i_2}, x_{i_4}\rangle$.

As a result, we get an inequality which is similar to~(\ref{eq:12}).
It expresses the fact that the points $x_{i_3}$ and $x_{i_4}$ are contained in one of
the two closed halfspaces determined by the plane containing the face $\langle x_{i_0}, x_{i_1}, x_{i_2}\rangle$ of $P$.
It is easy to understand that there will be $8\times 3=24$ of such inequalities.
Together they guarantee the convexity of~$P$.
We call them the second group of necessary conditions for the affine equivalence of~$P$ and $P'$.

\subsubsection*{The third group of necessary conditions}
This group of necessary conditions is constructed from~$P'$ in the same way as the first group was constructed from~$P$.
In other words, the third group of necessary conditions consists of four relations, each of which is obtained from 
(\ref{eq:7})--(\ref{eq:10}) by replacing every $d_{ij}$ by $d'_{ij}$ and every $\delta_{ij}$ by $\delta'_{ij}$.

\subsubsection*{The fourth group of necessary conditions}
This group is constructed from~$P'$ in the same way as the second group was constructed from~$P$.
It consists of 24 inequalities. All together they guarantee that~$P'$ is convex.

\subsubsection*{The fifth group of necessary conditions}
The relations of this group directly expresses the fact that $P$ and $P'$ are affine equivalent.
They are constructed as follows.

Select any of the 12 edges of $P$. Denote it by $\langle x_{i_0}, x_{i_1}\rangle$.
By $\langle x_{i_0}, x_{i_1}, x_{i_2}\rangle$ and $\langle x_{i_0}, x_{i_1}, x_{i_3}\rangle$
denote the two faces of $P$ incident to the edge $\langle x_{i_0}, x_{i_1}\rangle$.
According to our definition of an octahedron, the tetrahedron 
$\langle x_{i_0}, x_{i_1}, x_{i_2}, x_{i_3}\rangle$ has nonzero 3-volume,
$\operatorname{vol}\langle x_{i_0}, x_{i_1}, x_{i_2}, x_{i_3}\rangle$,
which is related to 3-volume of the tetrahedron $\langle x'_{i_0}, x'_{i_1}, x'_{i_2}, x'_{i_3}\rangle$,
$\operatorname{vol}\langle x'_{i_0}, x'_{i_1}, x'_{i_2}, x'_{i_3}\rangle$,
by the well-known formula 
$\operatorname{vol}(x'_{i_0}, x'_{i_1}, x'_{i_2}, x'_{i_3})=|\det A| \operatorname{vol}(x_{i_0}, x_{i_1}, x_{i_2}, x_{i_3})$.
Squaring the latter formula and using (\ref{eq:1}), we get 
\begin{equation*}
\operatorname{cm}(x'_{i_0}, x'_{i_1}, x'_{i_2}, x'_{i_3})= (\det A)^2\operatorname{cm}(x_{i_0}, x_{i_1}, x_{i_2}, x_{i_3})
\end{equation*}
or, which is the same,
\begin{equation}
\left|
\begin{array}{ccccc}
0 & 1                        & 1                        & 1                        & 1                            \\
1\vphantom{\int\limits^1} & 0        & {d'}^2_{\!\!i_{0}i_{1}}  & {d'}^2_{\!\!i_{0}i_{2}}  & {d'}^2_{\!\!i_{0}i_{3}}      \\
1\vphantom{\int\limits^1} & {d'}^2_{\!\!i_{0}i_{1}}  & 0              & {d'}^2_{\!\!i_{1}i_{2}}  & {d'}^2_{\!\!i_{1}i_{3}}      \\
1\vphantom{\int\limits^1} & {d'}^2_{\!\!i_{0}i_{2}}  & {d'}^2_{\!\!i_{1}i_{2}}  & 0      & {\delta'}^2_{\!\!i_{2}i_{3}} \\
1\vphantom{\int\limits^1} & {d'}^2_{\!\!i_{0}i_{3}}  & {d'}^2_{\!\!i_{1}i_{3}}  & {\delta'}^2_{\!\!i_{2}i_{3}}     & 0 
\end{array}
\right|
=\alpha
\left|
\begin{array}{ccccc}
0 & 1             & 1              & 1                   & 1                    \\
1\vphantom{\int\limits^1} & 0             & d^2_{i_0i_1}   & d^2_{i_0i_2}        & d^2_{i_0i_3}         \\
1\vphantom{\int\limits^1} & d^2_{i_0i_1}  & 0              & d^2_{i_1i_2}        & d^2_{i_1i_3}         \\
1\vphantom{\int\limits^1} & d^2_{i_0i_2}  & d^2_{i_1i_2}   & 0                   & \delta^2_{i_2i_3}    \\
1\vphantom{\int\limits^1} & d^2_{i_0i_3}  & d^2_{i_1i_3}   & \delta^2_{i_2i_3}   & 0              
\end{array}
\right|,
\label{eq:13}
\end{equation}
where $\alpha= (\det A)^2>0$.
All 12 equations forming the fifth group of necessary conditions 
may be obtained from (\ref{eq:13}) with a suitable choice of $x_{i_k}$, $k=0,1,2,3$.

\section{The main result}\label{s:4}
In Section~\ref{s:4}, we treat the relations which make up the five groups of necessary conditions for $P$ and $P'$
to be affinely equivalent to each other as algebraic equations and inequalities in seven unknowns  
$\delta_{05}$, $\delta_{14}$, $\delta_{23}$, $\delta'_{05}$, $\delta'_{14}$, $\delta'_{23}$, and $\alpha $, 
i.e., with respect to those variables whose values cannot be directly found from the natural developments~$R$ and $R'$ 
of $P$ and $P'$.
As we know from Section~\ref{s:3}, the coefficients of those algebraic equations and inequalities are expressed in 
terms of the variables whose values can be directly found from the natural developments $R$ and $R'$ of $P$ and $P'$, 
i.e., through $\{d_{ij}\}$, the set of the lengths of the edges $\langle x_i,x_j\rangle$ of $R$, and $\{d'_{ij}\}$, 
the set of the lengths of the edges $\langle x '_i,x'_j\rangle$ of $R'$.

The main result of this article is given by the following theorem:

\begin{thm}\label{thm:2}
Let $P$ and $P'$ be two convex octahedra in $\mathbb{R}^3$, and let $R$ and $R'$ be their natural developments;
then the following statements are equivalent:
\par
\emph{(a)} $P$ and $P'$ are affinely equivalent to each other;
\par
\emph{(b)} there are seven positive real numbers $\delta_{05}$, $\delta_{14}$, $\delta_{23}$, 
$\delta'_{05}$, $\delta'_{14}$, $\delta'_{23}$, and $\alpha$ 
such that all five groups of necessary conditions given in Section~\ref{s:3} are satisfied.
\end{thm} 

\begin{proof}
The implication (a) $\Rightarrow$ (b) has been demonstrated in Section \ref{s:3}.
One should take the lengths of the corresponding diagonals of $P$ as $\delta_{05}$, $\delta_{14}$, $\delta_{23}$, 
the lengths of the corresponding diagonals of $P'$ as $\delta'_{05}$, $\delta'_{14}$, $\delta'_{23}$, 
and the square of the determinant of the affine transformation $A$ such that $P'=A(P)$ as $\alpha$.

It remains to prove that (b) implies (a).
By $\widetilde{x}_i$, $i=0, \dots, 5$, denote the vertices of~$R$ which are enumerated in accordance with 
the special enumeration of the vertices of~$P$ introduced in Section~\ref{s:3} and depicted in Fig.~\ref{fig:1}.

Let us define a semimetric $\widetilde{\rho}$ on the set $\widetilde{V}=\{\widetilde{x}_0, \dots, \widetilde{x}_5\}$ 
by putting  
\begin{equation*}
\widetilde{\rho} (\widetilde{x}_i,\widetilde{x}_j)\stackrel{\textrm{def}}{=}
\begin{cases}
d_{ij}, \text{ if }~ \widetilde{x}_i,\widetilde{x}_j \text{ are connected by an edge in } R; \\
\delta_{ij}, \text{ otherwise}.
\end{cases}
\end{equation*}
Here $d_{ij}$ is the length of the edge $\langle \widetilde{x}_i,\widetilde{x}_j\rangle$ in~$R$, and
$\delta_{ij}$ is one of the positive real numbers $\delta_{05}$, $\delta_{14}$, $\delta_{23}$
whose existence is asserted in (b).

It follows from (b) that the relations (\ref{eq:7})--(\ref{eq:10}) are fulfilled. 
Thus, according to Theorem~\ref{thm:1}, $(\widetilde{V},\widetilde{\rho})$ embeds isometrically in $\mathbb{R}^3$.
By $f:\widetilde{V}\to\mathbb{R}^3$ denote one of such embeddings.

Let $\Delta$ be one of the eight triangles constituting the natural development~$R$.
Suppose $\widetilde{x}_i$, $\widetilde{x}_j$, and $\widetilde{x}_k$ are the vertices of $\Delta$. 
Since the second group of necessary conditions described in Section~\ref{s:3}
(and, in particular, the inequality (\ref{eq:11})) is assumed to be fulfilled, the set 
\begin{equation*}
\bigl(\underset{m=0}{\overset{5}{\bigcup}} \{f(\widetilde{x}_m)\}\bigr)\setminus
\bigl(\{f(\widetilde{x}_i)\}\cup\{f(\widetilde{x}_j)\}\cup\{f(\widetilde{x}_k)\}\bigr),
\end{equation*}
is contained in one of the two closed halfspaces determined by the plane containing 
the triangle $\langle f(\widetilde{x}_i), f(\widetilde{x}_j), f(\widetilde{x}_k)\rangle$.
Hence, $\langle f(\widetilde{x}_i), f(\widetilde{x}_j), f(\widetilde{x}_k)\rangle$ 
is a face of the convex hull of the points $f(\widetilde{x}_i)$, $i=0, \dots, 5$.
This implies that the convex hull of the points $f(\widetilde{x}_i)$, $i=0, \dots, 5$, 
is a convex octahedron whose vertices are provided with the standard numbering as in Fig.~\ref{fig:1}. 
Let us denote this octahedron by $X$.

By $\widetilde{x}'_i$, $i=0, \dots, 5$, denote the vertices of~$R'$ which are enumerated 
in such a way that $\widetilde{x}_i$ and $\widetilde{x}'_i$ correspond to each other according to
the combinatorial equivalence of $R$ and $R'$.
As before, we define a semimetric $\widetilde{\rho}'$ on the set $\widetilde{V}'=\{\widetilde{x}'_0, \dots, \widetilde{x}'_5\}$ 
by putting  
\begin{equation*}
\widetilde{\rho}' (\widetilde{x}'_i,\widetilde{x}'_j)\stackrel{\textrm{def}}{=}
\begin{cases}
d'_{ij}, \text{ if }~ \widetilde{x}'_i,\widetilde{x}'_j \text{ are connected by an edge in } R'; \\
\delta'_{ij}, \text{ otherwise}.
\end{cases}
\end{equation*}
Here $d'_{ij}$ is the length of the edge $\langle \widetilde{x}'_i,\widetilde{x}'_j\rangle$ in~$R'$, and
$\delta'_{ij}$ is one of the positive real numbers $\delta'_{05}$, $\delta'_{14}$, $\delta'_{23}$
whose existence is asserted in (b).

According to Theorem~\ref{thm:1} and the third group of necessary conditions, $(\widetilde{V}',\widetilde{\rho}')$ 
embeds isometrically in $\mathbb{R}^3$.
By $f':\widetilde{V}'\to\mathbb{R}^3$ denote one of such embeddings.

The forth group of necessary conditions implies that if 
$\widetilde{x}'_i$, $\widetilde{x}'_j$, and $\widetilde{x}'_k$ are the vertices
of any of the eight triangles constituting the natural development~$R'$
then $\langle f'(\widetilde{x}'_i), f'(\widetilde{x}'_j), f'(\widetilde{x}'_k)\rangle$ 
is a face of the convex hull of the points $f'(\widetilde{x}'_i)$, $i=0, \dots, 5$.
Thus, the convex hull of the points $f'(\widetilde{x}'_i)$, $i=0, \dots, 5$, 
is a convex octahedron whose vertices are provided with the standard numbering as in Fig.~\ref{fig:1}. 
Let us denote this octahedron by $X'$.

Note that both the points $f(\widetilde{x}_i)$, $i=0,\dots,3$ and the points $f'(\widetilde{x}'_i)$, $i=0, \dots,3$, 
are the vertices of nondegenerate tetrahedra.
Therefore, there is a unique affine transformation $A:\mathbb{R}^3\to\mathbb{R}^3$ such that $f'(\widetilde{x}'_i)=A(f(\widetilde{x}_i))$ for all $i=0,\dots, 3$.
According to~(b), the fifth group of necessary conditions is satisfied with some~$\alpha>0$.
Hence, $(\det A)^2=\alpha$.

Let us prove two more properties of~$A$, namely,
$f'(\widetilde{x}'_j)=A(f(\widetilde{x}_j))$, $j=4,5$.

Let $k=2 \mbox{ or } 3$, and $j=4 \mbox{ or } 5$. By definition, put
\begin{alignat*}{2}
&\Delta_{k,j} & \,=\, & \langle A(f(\widetilde{x}_0)), A(f(\widetilde{x}_1)), A(f(\widetilde{x}_k)),A(f(\widetilde{x}_j))\rangle,\\
&\Delta'_{k,j} & \,=\, & \langle f'(\widetilde{x}'_0), f'(\widetilde{x}'_1), f'(\widetilde{x}'_k), f'(\widetilde{x}'_j)\rangle,\\
&\delta_k&\, =\, & \langle A(f(\widetilde{x}_0)), A(f(\widetilde{x}_1)), A(f(\widetilde{x}_k))\rangle =
\langle f'(\widetilde{x}'_0), f'(\widetilde{x}'_1), f'(\widetilde{x}'_k)\rangle .
\end{alignat*}
Obviously, the tetrahedra $\Delta_{k,j}$ and $\Delta'_{k,j}$ share the common face $\delta_k$.
Moreover, 3-volumes of $\Delta_{k,j}$ and $\Delta'_{k,j}$ are equal to each other:
\begin{align*}
\operatorname{vol}(\Delta_{k,j})& =
\operatorname{vol}(A(f(\widetilde{x}_0)), A(f(\widetilde{x}_1)), A(f(\widetilde{x}_k)),A(f(\widetilde{x}_j))) \\
 & = |\det A| \operatorname{vol}(f(\widetilde{x}_0), f(\widetilde{x}_1), f(\widetilde{x}_k), f(\widetilde{x}_j))\\
 & = \sqrt{\alpha} \operatorname{vol}(f(\widetilde{x}_0), f(\widetilde{x}_1), f(\widetilde{x}_k), f(\widetilde{x}_j))\\
 & = \operatorname{vol}(f'(\widetilde{x}'_0), f'(\widetilde{x}'_1), f'(\widetilde{x}'_k), f'(\widetilde{x}'_j)) \\
 & =\operatorname{vol}(\Delta'_{k,j}).
\end{align*}
Therefore, the heights of $\Delta_{k,j}$ and $\Delta'_{k,j}$ 
which are treated as pyramids with the common base $\delta_k$ are equal to each other.
Denote this common height as $h_{k,j}$.

Recall that the affine hull of $\delta_k$ is denoted by $\mbox{aff\,}\delta_k$.
For $k=2, 3$ and $j=4, 5$, by $\tau_{k,j}$ denote the plane which is parallel to $\mbox{aff\,}\delta_k$
and lies at distance $h_{k,j}$ from $\mbox{aff\,}\delta_k$ in the closed halfspace which is bounded by 
$\mbox{aff\,}\delta_k$ and contains both $\Delta_{k,j}$ and $\Delta'_{k,j}$.
It follows from above that both points $A(f(\widetilde{x}_j))$, $f'(\widetilde{x}'_j)$ lie on $\tau_{k,j}$.

Let $i=0 \mbox{ or } 1$.
By $\sigma_i$ denote the affine hull of the three points $A(f(\widetilde{x}_i))=f'(\widetilde{x}'_i)$, 
$A(f(\widetilde{x}_2))=f'(\widetilde{x}'_2)$, and $A(f(\widetilde{x}_3))=f'(\widetilde{x}'_3)$.
Arguing as above, we see that the distance from either of the points $A(f(\widetilde{x}_{i+4}))$ and 
$f'(\widetilde{x}'_{i+4})$ to $\sigma_i$ is the same and these points lie in the same closed
halfspace bounded by $\sigma_i$.
Hence, $A(f(\widetilde{x}_{i+4}))$ and $f'(\widetilde{x}'_{i+4})$ lie on a plane parallel to $\sigma_i$.
Denote this plane by $\tau_i$.

Observe that both points $A(f(\widetilde{x}_4))$ and $f'(\widetilde{x}'_4)$
lie on each of the planes $\tau_{2,4}$, $\tau_{3,4}$ and $\tau_0$, and hence are 
contained in their intersection.
However, the planes $\tau_{2,4}$, $\tau_{3,4}$ and $\tau_0$ are parallel to the planes 
$\mbox{aff\,}\delta_2$, $\mbox{aff\,}\delta_3$ and $\sigma_0$ respectively, while
$\mbox{aff\,}\delta_2\cap \mbox{aff\,}\delta_3\cap \sigma_0=\{A(f(\widetilde{x}_0))\}=\{f'(\widetilde{x}'_0)\}$.
Hence, the intersection $\tau_{2,4}\cap \tau_{3,4}\cap \tau_0$ consists of exactly one point, 
and this point is $A(f(\widetilde{x}_4))=f'(\widetilde{x}'_4)$.

Similarly, we check that $\tau_{2,5}\cap \tau_{3,5}\cap \tau_1=\{A(f(\widetilde{x}_5))\}=\{f'(\widetilde{x}'_5)\}$
which yields $A(f(\widetilde{x}_5))=f'(\widetilde{x}'_5)$.

Thus, assuming (b) to be satisfied, we were able to isometrically embed the natural developments $R$ and $R'$ 
as convex affinely equivalent octahedra $X$ and $X'$ (since there is an affine transformation $A:\mathbb{R}^3\to\mathbb{R}^3$ 
such that $f'(\widetilde{x}'_j)=A(f(\widetilde{x}_j))$ for all $i=0,\dots, 5$).
On the other hand, $R$ is the natural development of both $P$ and $X$.
Thus, the Cauchy rigidity theorem implies that $P$ and $X$ are congruent.
Similarly, $P'$ and $X'$ are congruent to each other. Hence, $P$ and $P'$ are affinely equivalent, and  (b) implies (a).
\end{proof}

\section{Concluding remarks}\label{s:5}

{\scriptsize$\bullet$} In Problem~\ref{prb:1}, it would be more natural to ask about the projective (rather than affine) equivalence 
of convex octahedra.
The point is that both the property ``to be convex'' and the property ``to be a polyhedron'' are 
projective invariant.
Unfortunately, some arguments presented in this article are not applicable to the study of the projective equivalence of octahedra.
In fact, the equations from the fifth group of necessary conditions obtained in Section~\ref{s:3} 
are not valid for projective transformations.
Here we need qualitatively new ideas.

{\scriptsize$\bullet$} The Cauchy rigidity theorem is the standard of elegance.
Against this background, our Theorem~\ref{thm:2} looks especially clumsy.
It is desirable to find a solution to Problem~\ref{prb:1} as elegant as the Cauchy rigidity theorem.

\section{Acknowledgment}\label{s:6}

The study was carried out within the framework of the state contract
of the Sobolev Institute of Mathematics (project FWNF-2022-0006).

\end{document}